\documentclass[11pt]{amsart}

\usepackage{amsfonts}
\usepackage{amsmath,amssymb,amsthm}
\usepackage{enumerate,graphicx}
\usepackage{subcaption}
\usepackage[square,numbers]{natbib}
\usepackage{MnSymbol}
\usepackage{mathrsfs,bbm}

\newtheorem{theorem}{Theorem}[section]

\newtheorem{proposition}[theorem]{Proposition}
\newtheorem{corollary}[theorem]{Corollary}
\newtheorem{lemma}[theorem]{Lemma}

\theoremstyle{definition}

\newtheorem{definition}[theorem]{Definition}
\newtheorem{example}[theorem]{Example}

\newtheorem{remark}[theorem]{Remark}

\begin{document}

\title{Entropy bifurcation of neural networks on Cayley trees} 


\author{Jung-Chao Ban}
\address[Jung-Chao Ban]{Department of Applied Mathematics, National Dong Hwa University, Hualien 97401, Taiwan, ROC.}
\email{jcban@gms.ndhu.edu.tw}

\author[Chih-Hung Chang]{Chih-Hung Chang*}
\thanks{*Author to whom any correspondence should be addressed.} 
\author{Nai-Zhu Huang}
\address[Chih-Hung Chang and Nai-Zhu Huang]{Department of Applied Mathematics, National University of Kaohsiung, Kaohsiung 81148, Taiwan, ROC.}
\email{chchang@nuk.edu.tw; naizhu7@gmail.com}

\date{January 15, 2018}

\baselineskip=1.5\baselineskip

\begin{abstract}
It has been demonstrated that excitable media with a tree structure performed better than other network topologies, it is natural to consider neural networks defined on Cayley trees. The investigation of a symbolic space called tree-shift of finite type is important when it comes to the discussion of the equilibrium solutions of neural networks on Cayley trees. Entropy is a frequently used invariant for measuring the complexity of a system, and constant entropy for an open set of coupling weights between neurons means that the specific network is stable. This paper gives a complete characterization of entropy spectrum of neural networks on Cayley trees and reveals whether the entropy bifurcates when the coupling weights change.
\end{abstract}

\maketitle

\section{Introduction}

The human brain has recently been revealed as a system exhibiting traces of criticality; the corresponding spatiotemporal patterns are fractal-like. Gollo \emph{et al.}~\cite{GKC-SR2013} infer that criticality may arise from balanced dynamics within individual neurons. Neural networks have been developed to mimic brain behavior for the past few decades; they are widely applied in many disciplines such as signal propagation between neurons, deep learning, image processing, and information technology \cite{AA-N2003,B-FTML2009,CGM+-NNITo1992,Fuk-NN2013}. Chernihovskyi \emph{et al.}~\cite{CMM+-JCN2005} implement cellular neural networks on simulating nonlinear excitable media and develop a relevant device to predict epileptic seizures.
The overwhelming majority of neural network models adopts an $n$-dimensional lattice as the network's topology. Gollo \emph{et al.}~\cite{GKC-PCB2009,GKC-PRE2012,GKC-SR2013} propose a neural network with a tree structure; excitable media with a tree structure performed better than other network topologies since it attains larger dynamic range (cf.~\cite{AC-PRE2008,KC-NP2006,LSR-PRL2011}). It is of interest to ask the following problem.

\noindent \textbf{Problem 1.} How to measure the complexity of a tree structure neural network?

Alternatively, it is important to know how much information the neural network could store. On the other hand, it is of interest to know whether a neural network ``avalanches", which means such a network is sensitive. More precisely, some small modification of parameters could lead to tremendously different dynamics such as the exponential decay of storage of information. One of the most frequently studied neural networks is the Hopfield neural network consisting of locally coupled neurons, in which the behavior of each neuron is represented by a differential equation. Beyond being essential for understanding the dynamics of differential equations, the investigation of equilibrium solutions is related to elucidating the long-term memory of brain. Whenever there are only finitely many equilibrium solutions, the investigation of equilibrium solutions is then equivalent to studying shift spaces in symbolic dynamical systems.

A one-dimensional shift space is a set consisting of right-infinite or bi-infinite words which avoid words in a so-called \emph{forbidden set} $\mathcal{F}$ and is denoted by $\mathsf{X}_{\mathcal{F}}$. A shift space $\mathsf{X}_{\mathcal{F}}$ is called a shift of finite type (SFT) if $\mathcal{F}$ is a finite set. A significant invariant of shift spaces is the topological entropy, which reflects how much information a network can store. While there is an explicit formula for the entropy of $1$-d SFTs, there is no algorithm for the computation of the topological entropy of multidimensional SFTs so far (cf.~\cite{LM-1995, MP-ETDS2013, MP-SJDM2013, HM-AoM2010}).

Aubrun and B\'{e}al \cite{AB-TCS2012, AB-TCS2013} introduce the notion of tree-shifts, which are shift spaces defined on Cayley trees, and then study the classification theory up to conjugacy, languages, and its application to automaton theory. It is noteworthy that such tree-shifts constitute an intermediate class in between one-sided and multidimensional shifts. Ban and Chang \cite{BC-2017, BC-N2017} propose an algorithm for computing the entropy of a tree-shift of finite type (TSFT). The computation of the rigorous value of entropy is tricky due to the double exponential growth rate of the patterns for a TSFT (see Section 2 for more details).

For the case where TSFTs come from the equilibrium solutions of neural networks (on Cayley trees), the forbidden sets are constrained by the so-called \emph{separation property}; this makes the entropy spectrum discrete (Theorem \ref{thm:CTNN-entropy-set}). Elucidating the phenomenon of ``neural avalanches'' is related to the study of entropy bifurcation or entropy minimality problems. It is known that an irreducible $\mathbb{Z}^1$ SFT is entropy minimal; that is, any proper subshift $Y \subset X$ has smaller entropy than that of an irreducible SFT $X$. For $r \geq 2$, every $\mathbb{Z}^r$ SFT having the mixing property called \emph{uniform filling property} is entropy minimal while there is a non-trivial \emph{block gluing} $\mathbb{Z}^r$ SFT which is not entropy minimal. Readers are referred to \cite{BPS-TAMS2010,LM-1995,QS-ETDS2003} for more details. Proposition \ref{prop:entropy-region-W-equation} gives an explicit formula for the coupling weights between neurons which make CTNNs entropy minimal, and the entropy bifurcation diagram is revealed (Figure \ref{fig:entropy-diagram-general-case}). A remarkably novel phenomenon is that the entropy of a CTNN with the nearest neighborhood is either $0$ or $\ln d$, where $d$ is the number of children of each node.

The structure of this paper is as follows. Section 2 introduces the notion of tree-shifts and the algorithm for the computation of entropy of TSFTs. Section 3, aside from demonstrating how the investigation of the equilibrium solutions of CTNNs relates to the discussion of TSFTs, studies the learning problem of CTNNs; the necessary and sufficient condition of the forbidden sets of TSFTs corresponding to CTNNs is revealed. After demonstrating the discreteness of entropy spectrum of CTNNs, the entropy minimality problem is affirmatively solved in Section 3. Conclusion and discussion are given in Section 4.

\section{Symbolic Dynamics on Cayley Trees} \label{sec:dynamics-cayley-tree}

This section recalls some definitions and results of symbolic dynamics on Cayley trees. A novel phenomenon about the entropy spectrum of tree-shifts of finite type is demonstrated herein.

\subsection{Definitions and Notations}
A Cayley tree, roughly speaking, is a graph without cycles. Two kinds of Cayley trees are mostly discussed: rooted Cayley trees and bi-rooted Cayley trees. A rooted $d$-ary Cayley tree (Figure \ref{fig:GMTS-binary-tree}) can be seen as a directed graph such that the outdegree of each vertex is $d$ while a bi-rooted $d$-ary tree (also known as Bethe lattice, see \cite{Rozikov-2013} for more details) is an undirected graph such that the degree of each vertex is $d+1$. In this paper, we focus on the rooted Cayley tree for clarity, and the discussion can extend to the Bethe lattice. In the rest of this elaboration, we refer to rooted Cayley tree as Cayley tree unless otherwise stated.

Alternatively, a Cayley tree of order $d$ is a free semigroup $\Sigma^*$ generated by $\Sigma = \{g_1, g_2, \ldots, g_d\}$, where $d \in \mathbb{N}$. A \emph{labeled tree} $t$ over a finite alphabet $\mathcal{A}$ is a function from $\Sigma^*$ to $\mathcal{A}$; a \emph{node} of a labeled tree is an element of $\Sigma^*$, and the identity element relates to the \emph{root} of the tree. Suppose $x = x_1 x_2 \ldots x_i, y = y_1 y_2 \ldots y_j \in \Sigma^*$ are nodes of a tree, we say that $x$ is a \emph{prefix} of $y$ if and only if $i \leq j$ and $x_k = y_k$ for $1 \leq k \leq i$, and $xy = x_1 \cdots x_i y_1 \cdots y_j$ means the concatenation of $x$ and $y$. A subset $L \subset \Sigma^*$ is called \emph{prefix-closed} if the prefix of every element of $L$ belongs to $L$. A \emph{pattern} is a function $u: L \to \mathcal{A}$ with \emph{support} $L$ and is called an \emph{$(n+1)$-block} if its support $L = x \Delta_n$ for some $x \in \Sigma^*$, where $\Delta_n = \{y = y_1 y_2 \cdots y_n: y_i \in \Sigma \bigcup \{e\}\}$.

\begin{figure}
\centering
\begin{subfigure}{0.45\textwidth}
\begin{center}
\includegraphics[scale=0.8,page=1]{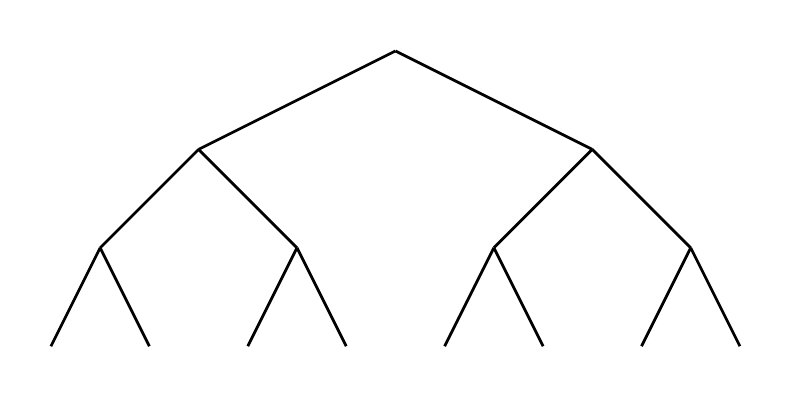}
\end{center}
\caption{Binary Cayley tree.}
\label{fig:GMTS-binary-tree}
\end{subfigure}
\begin{subfigure}{0.45\textwidth}
\begin{center}
\includegraphics[scale=0.8,page=2]{TreeNN-EntropyMinimal-20171215-pics}
\end{center}
\caption{Allowed pattern (left) and forbidden pattern (right).}\label{fig:GMTS-forbidden}
\end{subfigure}

\caption{A (rooted) binary Cayley tree is seen as a directed graph without cycles such that the outdegree of every vertex is $2$; herein we omit the arrow of the edge as seen in (A). A golden mean shift on binary Cayley tree is a tree-shift $\mathsf{X}_{\mathcal{F}}$ such that no consecutive $1$'s is allowed. Allowed and forbidden patterns are presented in (B).}\label{fig:golden-mean-tree-shift}
\end{figure}

Suppose that $u$ is a pattern and $t$ is a labeled tree. Let $s(u)$ denote the support of $u$. We say that $u$ is accepted by $t$ if there exists $x \in \Sigma^*$ such that $u_y = t_{xy}$ for every node $y \in s(u)$. In this case, we say that $u$ is a pattern of $t$ rooted at the node $x$. A tree $t$ is said to \emph{avoid} $u$ if $u$ is not accepted by $t$; otherwise, $u$ is called an \emph{allowed pattern} of $t$ (see Figure \ref{fig:GMTS-forbidden} for instance).

We denote by $\mathcal{T}$ (or $\mathcal{A}^{\Sigma^*}$) the set of all labeled trees on $\mathcal{A}$. The shift transformation $\sigma: \Sigma^* \times \mathcal{T} \to \mathcal{T}$ is defined by $(\sigma_w t)_x = t_{wx}$ for all $w, x \in \Sigma^*$. Given a collection of patterns $\mathcal{F}$, let $\mathsf{X}_{\mathcal{F}}$ denote the set of trees avoiding any element of $\mathcal{F}$. A subset $X \subseteq \mathcal{T}$ is called a \emph{tree-shift} if $X = \mathsf{X}_{\mathcal{F}}$ for some $\mathcal{F}$. We say that $\mathcal{F}$ is \emph{a set of forbidden patterns} (or \emph{a forbidden set}) of $X$. A tree-shift $X = \mathsf{X}_{\mathcal{F}}$ is called a \emph{tree-shift of finite type} (TSFT) if the forbidden set $\mathcal{F}$ is finite; we say that $\mathsf{X}_{\mathcal{F}}$ is a \emph{Markov tree-shift} if $\mathcal{F}$ consists of two-blocks. Ban and Chang \cite{BC-TAMS2017} demonstrate that every TSFT can be treated as a Markov tree-shift after recoding, which extends a classical result in symbolic dynamical systems.

\begin{proposition}[See \cite{BC-TAMS2017}]\label{prop:TSFT-is-vertex-shift-is-Markov}
Every tree-shift of finite type is conjugated to a Markov tree-shift.
\end{proposition}

Proposition \ref{prop:TSFT-is-vertex-shift-is-Markov} indicates that the investigation of Markov tree-shifts is essential for characterizing TSFTs. For the rest of this paper, a TSFT is referred to as a Markov tree-shift unless otherwise stated.

\subsection{Entropy of tree-shifts}

An important invariant of shift spaces is topological entropy, which measures the growth rate of the number of the admissible patterns. Such an invariant reflects the complexity on its own right, we refer readers to \cite{ASY-1997} for more details. The \emph{entropy} of tree-shifts is defined as
\begin{equation}\label{eq:entropy-tree-shift}
h(X)=\limsup_{n \rightarrow \infty} \frac{\ln^2 |B_{n}(X)|}{n},
\end{equation}%
where $B_{n}(X)$ is the collection of $n$-blocks of $X$, $|B_{n}(X)|$ means the cardinality of $B_{n}(X)$, and $\ln^2 = \ln \circ \ln $. Ban and Chang indicate that the limit $h(X)=\lim_{n\rightarrow \infty} \ln^2 |B_{n}(X)|/n$ exists if $X$ is a TSFT and $h(X) \in \{0, \ln 2\}$ for each TSFT $X$ when $d=2$ \cite{BC-2017,BC-N2017}; furthermore, a sufficient condition for positive entropy is revealed \cite{BC-JMP2017}.
For the computation of entropy, Ban and Chang introduce the notion of \emph{system of nonlinear recursive equations}.

\begin{definition}\label{def:SNRE}
Given $k \in \mathbb{N}$, we say that a sequence $\{\alpha_{1;n}, \alpha_{2;n}, \ldots, \alpha_{k;n}\}_{n \in \mathbb{N}}$ forms a \emph{system of nonlinear recursive equations (SNRE)} of degree $(d, k)$ if
\begin{equation*}
\alpha_{i;n} = F_i(n) \quad \text{for} \quad n \geq 2, 1 \leq i \leq k,
\end{equation*}%
with initial condition $\alpha_{i;1} \in \mathbb{N}$ for $1 \leq i \leq k$, where
$$
F_i(n) = \sum_{c_1 + c_2 + \cdots + c_k = d} r_{i; c_1, \ldots, c_k} \alpha_{1; n-1}^{c_1} \alpha_{2; n-1}^{c_2} \cdots \alpha_{k; n-1}^{c_k}
$$
with $r_{i; c_1, \ldots, c_k} \in \mathbb{Z}^+$.
\end{definition}

Let $F = \{F_1, F_2, \ldots, F_k\}$ be defined in Definition \ref{def:SNRE}. We also say that the sequence $\{\alpha_{1;n}, \alpha_{2;n}, \ldots, \alpha_{k;n}\}_{n \geq \mathbb{N}}$ is defined by $F$. For simplicity, $F$ is called the SNRE corresponding to $X$. Suppose that $F$ is given. For $1 \leq i \leq k$, we define the \emph{indicator vector} $v_{F_i}$ of $F_i$ as $v_{F_i} = (r_{i; c_1, \ldots, c_k})$. Note that the indicator vector $v_{F_i}$ is unique up to permutation. For the convenience, we represent the indicator vector with respect to the lexicographic order. The matrix
$I_F = \begin{pmatrix}
v_{F_1} \\
v_{F_2} \\
\vdots \\
v_{F_k}
\end{pmatrix}$
is called the \emph{indicator matrix} of $F$. For example, suppose that the sequence $\{\alpha_{1;n}, \alpha_{2;n}\}_{n \geq \mathbb{N}}$ forms the SNRE
\begin{equation}\label{eq:example-SNRE}
\left\{ 
\begin{array}{l}
\alpha_{1;n} = F_1 = \alpha_{1;n-1}^{2} +\alpha_{2;n-1}^{2}, \\ 
\alpha_{2;n} = F_2 = 2 \alpha_{1;n-1} \alpha_{2;n-1}, \\ 
\alpha_{1;1} = \alpha_{2;1} = 1.%
\end{array}%
\right.
\end{equation}
Then the corresponding indicator matrix is
\begin{equation*}
I_{F}= \begin{pmatrix}
1 & 0 & 1 \\ 
0 & 2 & 0
\end{pmatrix}.
\end{equation*}

Suppose $X$ is a TSFT over $\mathcal{A} = \{a_1, a_2, \ldots, a_k\}$. Let
$$
X_i = \{t \in X: t_{\epsilon} = a_i\} \quad \text{and} \quad \gamma_{i;n} = |B_n(X_i)|
$$
for $1 \leq i \leq k$. It follows immediately that $\{\gamma_{1;n}, \gamma_{2;n}, \ldots, \gamma_{k;n}\}_{n \in \mathbb{N}}$ forms an SNRE. Furthermore, every SNRE of degree $(d, k)$ can be realized via a TSFT (cf.~\cite{BC-N2017}). Let $F = \{F_1, \ldots, F_k\}$ be the representation of the SNRE of $X$. A subsystem called a \emph{reduced system of nonlinear recursive equations} of $F$ is defined as follows.

\begin{definition}\label{def:RSNRE}
Suppose $X$ is a TSFT, and $F$ is the SNRE corresponding to $X$ with indicator matrix $I_F$. We call $E$ a \emph{reduced system of nonlinear recurive equations} of $F$ if $E$ is an SNRE such that $I_{E}$ is a binary matrix satisfying the following conditions.
\begin{enumerate}[\bf (i)]
\item $I_{E} \leq I_{F}$;
\item there is exactly one nonzero entry in each row of $I_{E}$;
\item the initial condition of the sequence defined by $E$ is the same as the one defined by $F$.
\end{enumerate}
Herein, two matrices $A, B \in \mathbb{Z}^{m \times n}$ with $A \leq B$ means that $A(i, j) \leq B(i, j)$ for $1 \leq i \leq m, 1 \leq j \leq n$.
\end{definition}

Beyond defining the indicator matrix of an SNRE, a $k\times k$ nonnegative integral matrix $M_E$, called the \emph{weighted adjacency matrix}, of a reduced SNRE $E$, is defined as
\begin{equation}\label{eq:weighted-matrix}
M_E (i, j) = \max \{m: \alpha_{j;n-1}^m|E_i\}, \quad 1 \leq i, j \leq k.
\end{equation}
For example, consider the SNRE $F=\{F_i\}_{i=1}^{2}$ defined in \eqref{eq:example-SNRE} with indicator matrix
\begin{equation*}
I_{F}= \begin{pmatrix}
1 & 0 & 1 \\ 
0 & 2 & 0
\end{pmatrix}.
\end{equation*}
Then a reduced SNRE $E$ of $F$ with indicator matrix
\begin{equation*}
I_{E} = \begin{pmatrix}
1 & 0 & 0 \\ 
0 & 1 & 0
\end{pmatrix}
\end{equation*}
defines a sequence $\{\beta_{1;n}, \beta_{2;n}\}_{n \in \mathbb{N}}$ as follows.
\begin{equation*}
\left\{ 
\begin{array}{l}
\beta_{1;n} = E_1 = \beta_{1;n-1}^{2}, \\ 
\beta_{2;n} = E_2 = \beta_{1;n-1} \beta_{2;n-1}, \\ 
\beta_{1;1} = \beta_{2;1} = 1.
\end{array}%
\right.
\end{equation*}
Furthermore, the weighted adjacency matrix $M_E$ of $E$ is
$$
M_E = \begin{pmatrix}
2 & 0 \\
1 & 1
\end{pmatrix}.
$$

A symbol $a_i \in \mathcal{A}$ is called \emph{essential} if $\gamma_{i;n} \geq 2$ for some $n \in \mathbb{N}$; otherwise, we say that $a_i$ is \emph{inessential}. Suppose that, for a TSFT $X$ over $\mathcal{A}$, each symbol in $\mathcal{A}$ is essential.

\begin{theorem}[See \cite{BC-2017}]\label{thm:algorithm-entropy}
Let $X$ be a TSFT and let $F$ be the representation of the SNRE of $X$. If every symbol is essential, then 
\begin{equation}\label{eq:algorithm-entropy}
h(X) = \max \{\ln \rho _{M_E}: E \text{ is a reduced SNRE of } F\},
\end{equation}
where $M_E$ is the weighted adjacency matrix of $E$ and $\rho _{M_E}$ is the spectral radius of $M_E$.
\end{theorem}

Suppose that, for a TSFT $X$, there are some inessential symbols, say, $a_{p_1}, \ldots, a_{p_j}$. Ban and Chang demonstrate that Theorem \ref{thm:algorithm-entropy} still works provided, in \eqref{eq:algorithm-entropy}, $M_E$ is replaced by $M_E'$, where $M_E'$ is the matrix obtained by deleting all the rows and columns indexed by those inessential symbols. Readers are referred to \cite{BC-2017} for more details.

\begin{example}
Suppose that $d = 3, k = 4$. Let $X$ be a TSFT corresponds to the SNRE
\begin{equation*}
\left\{ 
\begin{array}{l}
\gamma_{1;n} = \gamma_{1;n-1} \gamma_{2;n-1} \gamma_{4;n-1} + \gamma_{4;n-1}^3, \\ 
\gamma_{2;n} = \gamma_{3;n-1} \gamma_{4;n-1}^2 + \gamma_{4;n-1}^3, \\ 
\gamma_{3;n} = \gamma_{1;n-1}^{2} \gamma_{2;n-1} + \gamma_{4;n}^3, \\ 
\gamma_{4;n} = \gamma_{4;n-1}^3, \\
\gamma_{i;1} = 1, 1 \leq i \leq 4.
\end{array}%
\right.
\end{equation*}
It is easily seen that $a_1, a_2, a_3$ are essential symbols and $a_4$ is inessential. The weighted adjacency matrix $M_E$ of the reduced SNRE $E$ which reaches the maximum in \eqref{eq:algorithm-entropy} is
$$
M_E = \begin{pmatrix}
1 & 1 & 0 & 1 \\ 
0 & 0 & 1 & 2 \\ 
2 & 1 & 0 & 0 \\ 
0 & 0 & 0 & 3
\end{pmatrix}.
$$
Since $a_4$ is inessential, we replace $M_E$ with
\begin{equation*}
M_E' = \begin{pmatrix}
1 & 1 & 0 \\ 
0 & 0 & 1 \\ 
2 & 1 & 0
\end{pmatrix}.
\end{equation*}
Theorem \ref{thm:algorithm-entropy} shows that the entropy of $X$ is $h(X) = \ln \rho_{M_E'} \approx \ln 1.839$, where $\rho_{M_E'}$ is the maximal root of $x^3 - x^2 - x - 1 = 0$.
\end{example}

\begin{proposition} \label{prop:essential-symbols-ln-d}
Suppose $X$ is a tree-shift of finite type and let $F$ be the representation of the SNRE of $X$. If every symbol is essential, then $h(X) = \ln d$.
\end{proposition}
\begin{proof}
It suffices to show that there exists a reduced SNRE $E$ of $F$ such that $h(E) = \ln d$ since $h(X) \leq \ln d$ (cf.~\cite{BC-N2017}). Let $E$ be a reduced SNRE of $F$. Then the weighted adjacency matrix $M_E$ satisfies $\sum\limits_{j=1}^k M_E(i, j) = d$ for $1 \leq i \leq d$. Since every symbol is essential, Theorem \ref{thm:algorithm-entropy} infers that $h(X) \geq \ln \rho_{M_E}$, where $\rho_{M_E}$ is the spectral radius of $M_E$. This completes the proof since $\rho_{M_E} = d$.
\end{proof}

Proposition \ref{prop:essential-symbols-ln-d} infers the rigidity of entropy since it is a constant ($\ln d$) whenever there is no inessential symbol. Let
\begin{equation}\label{eq:matrix-row-sum-less-than-equal-to-d}
D = \{M \in \mathcal{M}_{\ell \times \ell}(\mathbb{Z}^+): \sum_{q=1}^{\ell} M(p, q) \leq d \text{ for } 1 \leq p \leq \ell, \ell \leq k\}
\end{equation}
consist of nonnegative integral matrices whose dimension is less than or equal to $k$, and the summation of each row is less than or equal to $d$. Theorem \ref{thm:entropy-set-2symbol} illustrates a complete characterization of the entropy of TSFTs.

\begin{theorem}\label{thm:entropy-set-2symbol}
Let $H = \{h(X): X \text{ is a TSFT}\}$ be the entropy spectrum of TSFTs and $D$ is defined as in \eqref{eq:matrix-row-sum-less-than-equal-to-d}. Then
\begin{equation}\label{eq:entropy-spectrum-TSFT}
H = \{\ln \rho: \rho \text{ is the spectral radius of } M \in D\}.
\end{equation}
More specifically, $H = \{\ln \lambda: 1 \leq \lambda \leq d\}$ if $k = 2$.
\end{theorem}
\begin{proof}
We start with demonstrating the case where $k = 2$ to clarify our idea. If $a_1$ and $a_2$ are both essential symbols, Proposition \ref{prop:essential-symbols-ln-d} indicates that $h(X) = \ln d$. On the other hand, it is easily seen that $h(X) = 0$ provided $a_1$ and $a_2$ are both inessential. It remains to consider the case where exactly one symbol is essential.

Without loss of generality, we assume that $a_1$ is the essential symbol. It follows immediately from $a_2$ being inessential that, if $u \in \mathcal{F}$ such that $u_{\epsilon} = a_2$, $u_x = a_1$ for some $x = 1, 2, \ldots, d$. The SNRE of $X$ is then as follows.
$$
\left\{ 
\begin{aligned}
\gamma_{1;n} &= \sum_{c=0}^d \ell_c \gamma_{1;n-1}^c \gamma_{2;n-1}^{d-c}, \\ 
\gamma_{2;n} &= \gamma_{2;n-1}^d, n \geq 2, \\
\gamma_{i;1} &= 1, 1 \leq i \leq 2.
\end{aligned}
\right.
$$
Since $a_1$ is essential, there exists $c < d$ such that $\ell_c > 0$. Let $\overline{c} = \max\{c: \ell_c > 0\}$ and let $E$ be the representation of the following reduced SNRE.
$$
\left\{ 
\begin{aligned}
\beta_{1;n} &= \ell_{\overline{c}} \beta_{1;n-1}^{\overline{c}} \beta_{2;n-1}^{d-\overline{c}}, \\ 
\beta_{2;n} &= \beta_{2;n-1}^d.
\end{aligned}
\right.
$$
It follows from
$M_E = \begin{pmatrix}
\overline{c} & d - \overline{c} \\
0 & d
\end{pmatrix}$
and $M_E' = (\overline{c})$ that $h(X) = \ln \overline{c}$. This shows that $H \subseteq \{\ln \lambda: 1 \leq \lambda \leq d\}$.

Conversely, for $1 \leq c \leq d$, let $X$ be a TSFT correspond to the SNRE
$$
\left\{ 
\begin{aligned}
\gamma_{1;n} &= \gamma_{1;n-1}^c \gamma_{2;n-1}^{d-c} + \gamma_{2;n-1}^d, \\ 
\gamma_{2;n} &= \gamma_{2;n-1}^d, \\
\gamma_{i;1} &= 1, 1 \leq i \leq 2.
\end{aligned}
\right.
$$
Then $h(X) = \ln c$. The proof of $H = \{\ln \lambda: 1 \leq \lambda \leq d\}$ is thus complete.

Generally, the SNRE of $X$ is seen as
$$
\left\{ 
\begin{aligned}
&\gamma_{i;n} = \sum_{\mathbf{c}} \ell_{\mathbf{c}} \prod_{j=1}^k \gamma_{j;n-1}^{c_{i,j}}, \mathbf{c} = (c_{i,j}) \text{ satisfies } \sum_{j=1}^k c_{i, j} = d, n \geq 2, \\
&\gamma_{i;1} = 1, 1 \leq i \leq k.
\end{aligned}
\right.
$$
It suffices to consider the case where there is inessential symbol. Without loss of generality, we may assume that $a_1, \ldots, a_{\ell}$ are essential and $a_{\ell + 1}, \ldots, a_k$ are inessential for some $1 \leq \ell \leq k-1$. For $1 \leq i \leq k$, let $\overline{\mathbf{c}}_i = (c_{i, 1}, \ldots, c_{i, k})$ satisfy $\sum_{j=1}^{\ell} c_{i,j} \geq \sum_{j=1}^{\ell} c_{i,j}'$ for all $\mathbf{c}' = (c_{i, 1}', \ldots, c_{i, k}')$, and let $E$ be the SNRE
$$
\left\{ 
\begin{aligned}
&\beta_{i;n} = \ell_{\mathbf{c}_i} \prod_{j=1}^k \gamma_{j;n-1}^{c_{i,j}}, n \geq 2, \\
&\beta_{i;1} = 1, 1 \leq i \leq k.
\end{aligned}
\right.
$$
It is seen that $h(X) = \ln \rho$, where $\rho$ is the spectral radius of $M = (c_{i, j})_{1 \leq i, j \leq \ell}$. This elaborates
$$
H \subseteq \{\ln \rho: \rho \text{ is the spectral radius of } M \in D\}.
$$
The demonstration of $H \supseteq \{\ln \rho: \rho \text{ is the spectral radius of } M \in D\}$ is similar to the discussion above, thus it is omitted. This completes the proof.
\end{proof}

Suppose that $X = \mathsf{X}_{\mathcal{F}}$ is a TSFT over $\mathcal{A} = \{a_1, \ldots, a_k\}$. Recall that $\mathcal{F}$ consists of $2$-blocks. Let $\mathcal{B} = \mathcal{A}^{\Delta_1} \setminus \mathcal{F}$. For $\mathcal{A}' \subseteq \mathcal{A}$, set
$$
\mathcal{B}|_{\mathcal{A}'} = \{u \in \mathcal{B}: u_x \in \mathcal{A}' \text{ for } x \in \Delta_1\},
$$
and let $X|_{\mathcal{A}'}$ denote the subshift generated by $\mathcal{B}|_{\mathcal{A}'}$. We say that $\mathcal{A}'$ is essential if every symbol $a \in \mathcal{A}'$ is essential. This section ends with Corollary \ref{cor:loop-imply-ln-d}, which comes immediately from the proof of Theorem \ref{thm:entropy-set-2symbol} and is useful in the investigation of the entropy minimality problem of neural networks on Cayley trees.

\begin{corollary}\label{cor:loop-imply-ln-d}
Suppose that $X = \mathsf{X}_{\mathcal{F}}$ is a TSFT over $\mathcal{A} = \{a_1, \ldots, a_k\}$. Then $h(X) = \ln d$ if and only if $X|_{\mathcal{A}'}$ is nontrivial for some essential set $\mathcal{A}' \subseteq \mathcal{A}$. More specifically, when $k=2$, $h(X) = \ln d$ if and only if $\mathcal{A}$ is essential or the $2$-block $u \in \mathcal{B}$ with $u_x = a$ for $x \in \Delta_1$ and $a$ is essential.
\end{corollary}
\begin{proof}
Obviously, $X|_{\mathcal{A}'}$ being nontrivial for some essential set $\mathcal{A}' \subseteq \mathcal{A}$ infers that $h(X) = \ln d$. For the converse direction, it suffices to show the case where $k = 2$, the general cases can be derived analogously. If $a_1$ and $a_2$ are inessential, then $h(X) = 0$, which is a contradiction. Without loss of generality, we may assume that $a_1$ is essential.

If $a_2$ is essential, Proposition \ref{prop:essential-symbols-ln-d} demonstrates that $h(X) = \ln 2$. Otherwise, $h(X) = \ln d$ infers that the system
$$
\left\{ 
\begin{aligned}
\gamma_{1;n} &= \gamma_{1;n-1}^d, \\ 
\gamma_{2;n} &= \gamma_{2;n-1}^d,
\end{aligned}
\right.
$$
must be a reduced SNRE of the original system. This derives the desired result.
\end{proof}

\section{Neural Networks on Cayley Trees}

The overwhelming majority of models of neural networks are defined on $\mathbb{Z}^n$ lattice. While it is known that the characteristic shape of neurons is tree \cite{GKC-PCB2009}, this section considers neural networks defined on Cayley trees. A \emph{neural network on Cayley tree} (CTNN) is represented as
\begin{equation}\label{eq:cnn-tree}
\dfrac{d}{dt} x_w(t) = - x_w(t) + z + \sum_{v \in \mathcal{N}} a_v f(x_{wv}(t)), \quad w \in \Sigma^*,
\end{equation}
for some finite set $\mathcal{N} \subset \Sigma^*$ known as the neighborhood, $v \in \mathcal{N}$, and $t\geq 0$. Herein, $x_w(t) \in \mathbb{R}$ represents the internal status of neuron at $w$; the map $f(s)$ is called the \emph{output function} or \emph{activation function}, and $z \in \mathbb{R}$ is called the \emph{threshold}. The weighted parameters $A = (a_v)_{v \in \mathcal{N}}, a_v \in \mathbb{R},$ is called the \emph{feedback template}, and  Figure \ref{fig:TreeCNN} shows the connection of a binary CTNN with the nearest neighborhood. Equation \eqref{eq:cnn-tree} is derived by adopting Hopfield's neural network (\cite{Hopfield-PNASU1982}) on the Cayley tree. Normatov and Rozikov \cite{NR-MN2006} show that harmonic functions on Cayley trees, which is a discrete time version of \eqref{eq:cnn-tree}, are periodic with respect to normal subgroups of finite index. The present paper investigates the complexity of output solutions with respect to the output function
\begin{equation}\label{eq:piecewise-linear}
f(s)=\dfrac{1}{2}(|s+1|-|s-1|)  
\end{equation}
which is proposed by Chua and Yang \cite{CY-ITCS1988} and is widely applied to many disciplines such as signal propagation between neurons, pattern recognition, and self-organization.

\begin{figure}[tbp]
\begin{center}
\includegraphics[scale=0.8,page=3]{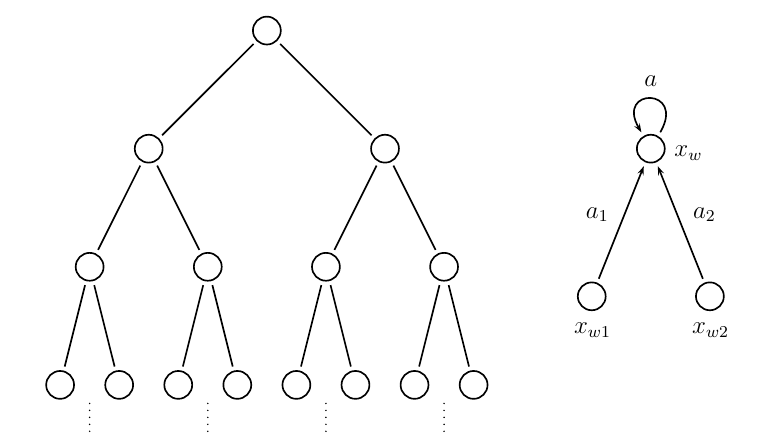}
\end{center}
\caption{A neural network with the nearest neighborhood defined on binary trees. In this case, the neighborhood $\mathcal{N} = \{\epsilon, 0, 1\}$ and $a = a_{\epsilon}$.}
\label{fig:TreeCNN}
\end{figure}

A \emph{mosaic solution} $x = (x_w)_{w \in \Sigma^*}$ of \eqref{eq:cnn-tree} is an equilibrium solution which satisfies $|x_w| > 1$ for all $w \in \Sigma^*$; its corresponding pattern $y = (y_w)_{w \in \Sigma^*} = (f(x_w))_{w \in \Sigma^*}$ is called a \emph{mosaic output pattern}. Since the output function \eqref{eq:piecewise-linear} is piecewise linear with $f(s)=1$ (resp.~$-1$) if $s \geq 1$ (resp.~$s \leq -1$), the output of a mosaic solution $x = (x_w)_{w \in \Sigma^*}$ must be an element in $\left\{ -1,+1\right\}^{\Sigma^*}$, which is why we call it a \emph{pattern}. Given a CTNN, we refer to $\mathbf{Y}$ as the output solution space; namely,
\begin{equation*}
\mathbf{Y} = \left\{ (y_w)_{w \in \Sigma^*}: y_w = f(x_w) \text{ and } (x_w)_{w \in \Sigma^*} \text{ is a mosaic solution of } \eqref{eq:cnn-tree} \right\} .
\end{equation*}

\subsection{Learning problem of neural networks on Cayley trees}

Learning problems (also called the inverse problems) are some of the most investigated topics in a variety of disciplines. From a mathematical point of view, determining whether a given collection of output patterns can be exhibited by a CTNN is essential for the study of learning problems. This section reveals the necessary and sufficient conditions for the capability of exhibiting the output patterns of CTNNs. The discussion is similar to the investigation in \cite{BC-NN2015,BCLL-JDE2009,Chang-ITNNLS2015}, thus we only sketch the key procedures of the learning problems of CTNNs with the nearest neighborhood, namely, $\mathcal{N} = \Sigma \bigcup \{e\}$, for the compactness and self-containedness of this paper.

A CTNN with the nearest neighborhood is realized as
\begin{equation}\label{eq:cnn-tree-nearest-nbd}
\dfrac{d}{dt} x_w(t) = - x_w(t) + z + a f(x_{w}(t)) + \sum_{i=1}^d a_i f(x_{wi}(t)),
\end{equation}
where $a, a_1, \ldots, a_d \in \mathbb{R}$ and $w \in \Sigma^*$. Considering the mosaic solution $x = (x_w)_{w \in \Sigma^*}$, the necessary and sufficient conditions for $y_w = f(x_w) = 1$ is
\begin{equation}\label{eq:cnn-state+}
a - 1 + z > - \sum_{i=1}^d a_i y_{wi}.
\end{equation}
Similarly, the necessary and sufficient conditions for $y_w = f(x_w) = -1$ is
\begin{equation}\label{eq:cnn-state-}
a - 1 - z > \sum_{i=1}^d a_i y_{wi}.
\end{equation}

Let
$$
V^n = \{ v \in \mathbb{R}^n : v = (v_1, \ldots, v_n), \text{ and } |v_i| = 1, 1 \leq i \leq n \}.
$$
Let $\alpha = (a_1, \ldots, a_d)$ represent the feedback template without the self-feedback parameter $a$. The basic set of admissible local patterns with the ``$+$" state in the parent neuron is denoted as
\begin{equation}
\widetilde{\mathcal{B}}_+( A, z) = \{v \in V^d: a - 1 + z > -\alpha \cdot v \},
\end{equation}
where ``$\cdot$" is the inner product in Euclidean space. Similarly, the basic set of admissible local patterns with the ``$-$" state in the parent neuron is denoted as
\begin{equation}
\widetilde{\mathcal{B}}_-( A, z) = \{v \in V^d: a - 1 - z > \alpha \cdot v \}.
\end{equation}
Furthermore, the admissible local patterns induced by $(A, z)$ can be denoted by
\begin{equation}
\mathcal{B}(A, z) = \mathcal{B}_+( A, z) \bigcup \mathcal{B}_-( A, z),
\end{equation}
where
\begin{align*}
\mathcal{B}_+( A, z) &= \{v: v_{\epsilon} = 1 \text{ and } (v_1, \ldots, v_d) \in \widetilde{\mathcal{B}}_+( A, z)\}, \\
\mathcal{B}_-( A, z) &= \{v: v_{\epsilon} = -1 \text{ and } (v_1, \ldots, v_d) \in \widetilde{\mathcal{B}}_-( A, z)\}.
\end{align*}
Note that $\mathcal{B}(A, z)$ consists of two-blocks over $\mathcal{A} = \{1, -1\}$. For simplicity, we omit the parameters $(A, z)$ and refer to $\mathcal{B}$ as $\mathcal{B}(A, z)$.

Suppose $U$ is a subset of $V^n$, where $n \geq 2 \in \mathbb{N}$. Let $U^c = V^n \setminus U$. We say that $U$ satisfies the \emph{linear separation property} if there exists a hyperplane $H$ that separates $U$ and $U^c$. More precisely, $U$ satisfies the separation property if and only if there exists a linear functional $g(z_1, z_2, \ldots, z_n) = c_1 z_1 + c_2 z_2 + \cdots + c_n z_n$ such that
$$
g(v) > 0 \quad \text{for} \quad v \in U \quad \text{and} \quad g(v) < 0 \quad \text{for} \quad v \in U^c.
$$
Figure \ref{fig:separation} interprets those $U \subset V^2$ satisfying the linear separation property.

\begin{figure}[tbp]
\begin{center}
\includegraphics[scale=0.8,page=4]{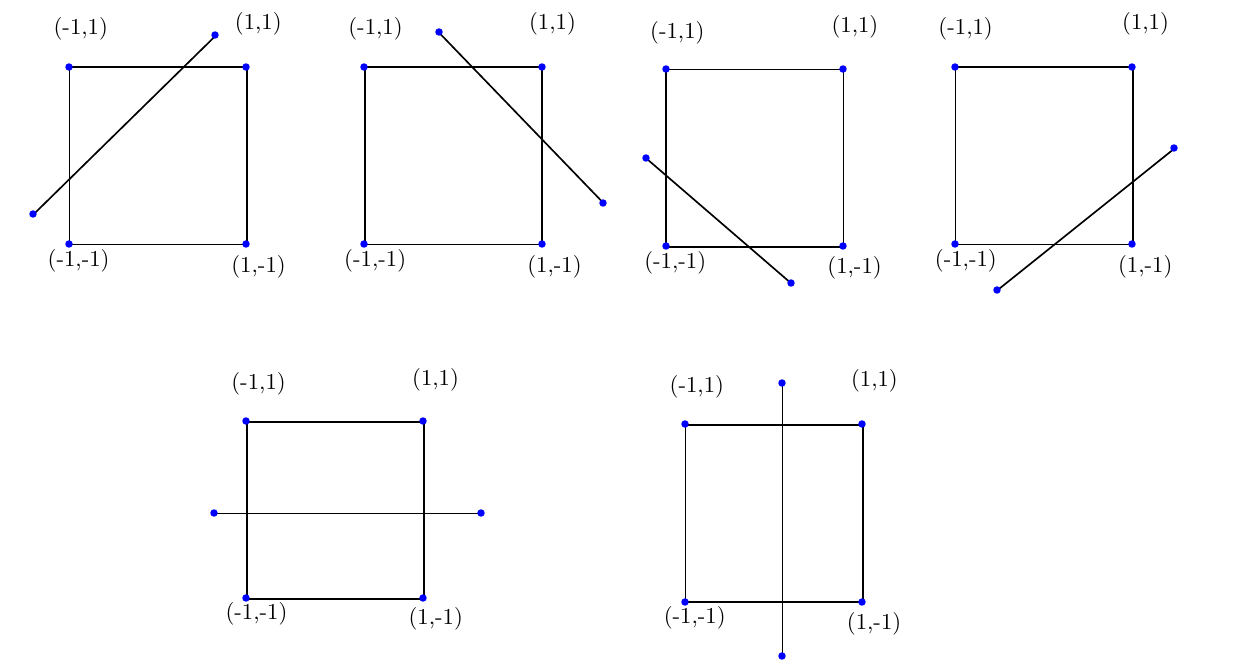}
\end{center}
\caption{Suppose $U \subseteq V^2 = \{-1, 1\}^2$ presents a set of allowable local patterns. Linear separation property infers that $U$ and $V^2 \setminus U$, geometrically, must be separated by a straight line. Hence, there are only $12$ choices of $U$ when $d = 2$.}
\label{fig:separation}
\end{figure}

Proposition \ref{prop:separation-property} elucidates the necessary and sufficient condition for the learning problems of CTNNs; such a property holds for arbitrary neighborhood $\mathcal{N}$ provided $\mathcal{N}$ is prefix-closed. The proof of Proposition \ref{prop:separation-property} is similar to the discussion in \cite{BC-NN2015}, thus it is omitted.

\begin{proposition}\label{prop:separation-property}
A collection of patterns $\mathcal{B} = \mathcal{B}_+ \bigcup \mathcal{B}_-$ can be realized in \eqref{eq:cnn-tree-nearest-nbd} if and only if either of the following conditions is satisfied:
\begin{enumerate}[\bf ({Inv}1)]
\item $-\mathcal{B}_+ \subseteq \mathcal{B}_-$ and $\mathcal{B}_-$ satisfies linear separation property;
\item $-\mathcal{B}_- \subseteq \mathcal{B}_+$ and $\mathcal{B}_+$ satisfies linear separation property.
\end{enumerate}
\end{proposition}

Let
\begin{equation}\label{eq:parameter-space}
\mathbb{R}^{d+2} = \{ (A, z) |\ A \in \mathbb{R}^{d+1}, z \in \mathbb{R} \}
\end{equation}
denote the parameter space. Theorem \ref{thm:partition} demonstrates that $\mathbb{R}^{d+2}$ can be partitioned into finitely equivalent sub-regions such that two sets of parameters induce identical basic sets of admissible local patterns if and only if they belong to the same partition in the parameter space. We skip the proof of Theorem \ref{thm:partition} for the compactness of this paper since the demonstration is similar to ths discussion in \cite{HJLL-IJBCASE2000}.

\begin{theorem}\label{thm:partition}
There exists a positive integer $K = K(d)$ and a unique collection of open subsets $\{P_i\}_{i=1}^K$ of the parameter space \eqref{eq:parameter-space} satisfying
\begin{enumerate}[\bf (i)]
  \item $\mathbb{R}^{d+2} = \bigcup\limits_{i=1}^K \overline{P}_k$;
  \item $P_i \bigcap P_j = \varnothing$ for all $i \neq j$;
  \item $\mathcal{B}(A, z) = \mathcal{B}(A', z')$ if and only if $(A, z), (A', z') \in P_i$ for some $1 \leq i \leq K$.
\end{enumerate}
Herein, $\overline{P}$ indicates the closure of $P$ in $\mathbb{R}^{d+2}$.
\end{theorem}

\begin{remark}
A straightforward examination asserts that, whenever a set of parameters $(A, z)$ is given, the output solution space $\mathbf{Y}$ is a Markov tree-shift since $\mathbf{Y} = \mathsf{X}_{\mathcal{F}}$, where $\mathcal{F} = \{-1, 1\}^{\Delta_1} \setminus \mathcal{B}(A, z)$.
\end{remark}

We consider the case where $d = 2$ as an example. Note that, whenever the parameters $a_1$ and $a_2$ are determined, \eqref{eq:cnn-state+} and \eqref{eq:cnn-state-} partition the $a$-$z$ plane into $25$ regions; the ``order" (i.e., the relative position) of lines $a - 1 +(-1)^{\ell} z = (-1)^{\ell} (a_1 y_{w1} + a_2 y_{w2})$, $\ell = 1, 2$, can be uniquely determined by the following procedures:
\begin{enumerate}[1)]
\item The signs of $a_1, a_2$ (i.e., the parameters are positive or negative).
\item The magnitude of $a_1, a_2$ (i.e., $|a_1| > |a_2|$ or $|a_1| < |a_2|$).
\end{enumerate}
This partitions $a$-$z$ plane into $8 \times 25 = 200$ sub-regions. According to Theorem \ref{thm:partition}, the parameter space $\mathcal{P}^4$ is partitioned into less than $200$ equivalent sub-regions.

\subsection{Entropy bifurcation of neural networks on Cayley trees} 

Suppose that, for each neuron of a neural network on Cayley tree, we substitute its output pattern $1$ (resp.~$-1$) with $+$ (resp.~$-$); then the output solution space $\mathbf{Y}$ of a CTNN is a Markov tree-shift over $\mathcal{A} = \{+, -\}$. We denote the TSFT $\mathbf{Y}$ by $\mathbf{Y}_{\mathcal{B}}$ when we want to emphasize the basic set of admissible patterns $\mathcal{B}$. This subsection investigates the entropy and the entropy bifurcation diagram of $\mathbf{Y}$.

We start with the following lemma, for which the proof can be done via straightforward elucidation, thus it is omitted.

\begin{lemma}\label{lem:zero-entropy-trivial case}
Suppose that $\mathbf{Y}_{\mathcal{B}}$ is an output solution space such that $u_{\epsilon} = v_{\epsilon}$ for all $u, v \in \mathcal{B}$. Then $h(\mathbf{Y}_{\mathcal{B}}) = 0$.
\end{lemma}

Based on Lemma \ref{lem:zero-entropy-trivial case}, we may assume that, for each basic set of admissible local patterns $\mathcal{B}$, there exist $u, v \in \mathcal{B}$ such that $u_{\epsilon} \neq v_{\epsilon}$. We call such a set of local patterns $\mathcal{B}$ \emph{nontrivial}; an output solution space $\mathbf{Y}_{\mathcal{B}}$ is called \emph{nontrivial} if its corresponding set of local patterns $\mathcal{B}$ is nontrivial.

\begin{theorem}\label{thm:CTNN-entropy-set}
Suppose that $\mathbf{Y}$ is an output solution spaces of \eqref{eq:cnn-tree-nearest-nbd}. Then $h(\mathbf{Y}) = 0$ or $\ln d$.
\end{theorem}
\begin{proof}
Lemma \ref{lem:zero-entropy-trivial case} suggests that we only need to consider nontrivial output solution spaces; that is, there exist $u, v \in \mathcal{B}$ such that $u_{\epsilon} = +$ and $v_{\epsilon} = -$. Proposition \ref{prop:essential-symbols-ln-d} demonstrates that $h(\mathbf{Y}) = \ln d$ if both symbols $+$ and $-$ are essential. It remains to consider the case where exactly one symbol is inessential.

Without loss of generality, we may assume that $-$ is inessential. In other words, if $u \in \mathcal{B}$ satisfies $u_{\epsilon} = -$, then $u_i = -$ for $1 \leq i \leq d$. Proposition \ref{prop:separation-property} shows that there exists $v \in \mathcal{B}$ such that $v_{\epsilon} = v_1 = \cdots = v_d = +$; Corollary \ref{cor:loop-imply-ln-d} indicates that $h(\mathbf{Y}_{\mathcal{B}}) = \ln d$. This completes the proof.
\end{proof}

The well-known entropy minimality problem investigates when the entropy of any proper subspace is strictly smaller than the entropy of the original shift space. For the case of CTNNs, the entropy minimality problem is equivalent to investigating under what condition $h(\mathbf{Y}_{\mathcal{B}'}) < h(\mathbf{Y}_{\mathcal{B}})$, where $\mathcal{B}'$ is obtained by deleting a pattern in $\mathcal{B}$. Furthermore, it follows from Theorem \ref{thm:CTNN-entropy-set} that the change of entropy is from $\ln d$ to $0$; in other words, simply removing a pattern from the basic set of admissible local patterns $\mathcal{B}$ makes significant influence to the original space. Equation \eqref{eq:W-shape} characterizes those parameters which make such a tremendous influence.

The discussion in the previous subsection shows that, once the parameters $a_1, \ldots, a_d$ are fixed, \eqref{eq:cnn-state+} and \eqref{eq:cnn-state-} partition the $a$-$z$ plane into $(2^d+1)^2$ regions. We encode these regions by $[p, q]$ for $0 \leq p, q \leq 2^d$ and denote the corresponding basic set of admissible local patterns as $\mathcal{B}_{[p, q]}$. More specifically, $\mathcal{B}_{[p, q]} = \mathcal{B}_{[p, q]; +} \bigcup \mathcal{B}_{[p, q]; -}$ which satisfies $|\mathcal{B}_{[p, q]; +}| = p$ and $|\mathcal{B}_{[p, q]; -}| = q$. For simplicity, we denote $\mathbf{Y}_{\mathcal{B}_{[p, q]}}$ by $\mathbf{Y}_{[p, q]}$. The following proposition comes immediately.

\begin{proposition}
Suppose that the parameters $a_1, \ldots, a_d$ are given. Then $\mathbf{Y}_{[p, q]} \cong \mathbf{Y}_{[q, p]}$ for $0 \leq p, q \leq 2^d$.
\end{proposition}
\begin{proof}
Since $\mathcal{B}_{[p, q]; +} = - \mathcal{B}_{[q, p]; -}$ and $\mathcal{B}_{[p, q]; -} = - \mathcal{B}_{[q, p]; +}$, the desired results is then derived.
\end{proof}

Suppose that the parameters $a_1, \ldots, a_d$ are given. A pair of parameters $(a, z)$ is called \emph{critical} if, for each $r > 0$, there exists $(a', z'), (a'', z'') \in B_r (a, z)$ such that $h(\mathbf{Y}_{\mathcal{B}'}) = \ln d$ and $h(\mathbf{Y}_{\mathcal{B}''}) = 0$, where $B_r (a, z)$ is the $r$-ball centered at $(a, z)$ and $\mathcal{B}' = \mathcal{B}(A', z'), \mathcal{B}'' = \mathcal{B}(A'', z''), A' = (a', a_1, \ldots, a_d)$, and $A'' = (a'', a_1, \ldots, a_d)$.

\begin{proposition}\label{prop:entropy-region-W-equation}
Suppose that the parameters $a_1, \ldots, a_d$ are given. Then $h(\mathbf{Y}_{[p, q]}) = 0$ if and only if
$$
\min\{p, q\} = 0 \text{ or } \max\{p, q\} = 1,
$$
where $0 \leq p, q \leq 2^d$. Furthermore, let $\ell$ be the index such that $|a_{\ell}| = \min \{|a_i|: 1 \leq i \leq d\}$. Then $(a, z)$ is critical if and only if
\begin{equation}\label{eq:W-shape}
a - 1 = \big||z| - |a_{\ell}|\big| - \sum_{i \neq \ell} |a_i|.
\end{equation}
\end{proposition}
\begin{proof}
Observe that the proof of Theorem \ref{thm:CTNN-entropy-set} demonstrates that $h(\mathbf{Y}_{[p, q]}) = 0$ if and only if $\min\{p, q\} = 0$ or $\max\{p, q\} = 1$. It remains to show that $(a, z)$ is critical if and only if $(a, z)$ satisfies \eqref{eq:W-shape}.

Let $C = \{\sum_{i=1}^d \ell_i a_i: \ell_i \in \{-1, 1\} \text{ for all } i\}$, and let
$$
K_1 = \max C \quad \text{and} \quad K_2 = \max C \setminus \{K_1\}
$$
be the largest and the second largest elements in $C$, respectively. A careful but straightforward verification asserts that $(a, z)$ is critical if and only if
$$
a - 1 = \left| |z| -\dfrac{K_1 - K_2}{2}\right| - \dfrac{K_1 + K_2}{2}.
$$
(See Figure \ref{fig:entropy-diagram-general-case} for more information.) The desired result follows from the fact that
$$
K_1 = \sum_{i=1}^d |a_i| \quad \text{and} \quad K_2 = \sum_{i \neq \ell} |a_i| - |a_{\ell}|.
$$
\end{proof}

\begin{figure}[tbp]
\begin{center}
\includegraphics[scale=0.8,page=5]{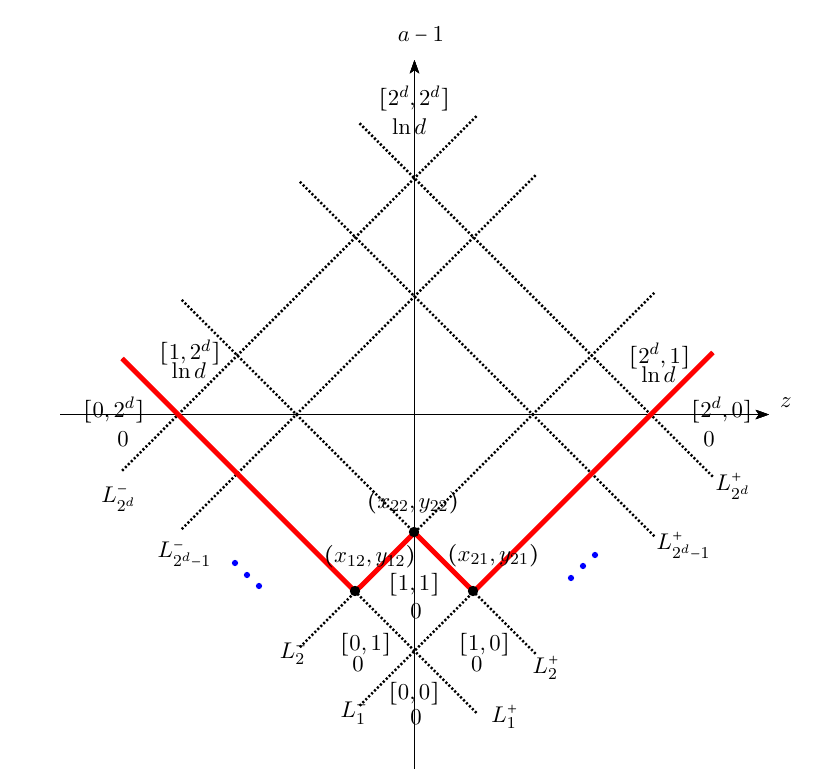}
\end{center}
\caption{Entropy bifurcation diagram of neural networks on Cayley trees. Whenever all the parameters are fixed except $a$ and $z$, the $a$-$z$ plane is partitioned into $(2^d + 1)^2$ regions. A CTNN has entropy $\ln d$ if and only if the parameter $(a, z)$ is above the red line.}
\label{fig:entropy-diagram-general-case}
\end{figure}

As the end of this section, we give the following example to clarify the investigation of entropy bifurcation diagrams of neural networks on the binary Cayley tree.

\begin{example}
A neural network on the binary Cayley tree is represented as
$$
\dfrac{d}{dt} x_w(t) = - x_w(t) + z + a f(x_{w}(t)) + a_1 f(x_{w1}(t)) + a_2 f(x_{w2}(t)),
$$
where $w \in \Sigma^*$ and $\Sigma = \{1, 2\}$. The necessary and sufficient conditions for $y_w = 1$ and $y_w = -1$ are
$$
a - 1 + z > - (a_1 y_{w1} + a_2 y_{w2})
\quad \text{and} \quad
a - 1 - z > a_1 y_{w1} + a_2 y_{w2},
$$
respectively. Suppose that $a_1, a_2$ satisfy $0 < -a_1 < a_2$. It follows from $a_1 - a_2 < -a_1 - a_2 < a_1 + a_2 < -a_1 + a_2$ that, whenever $a$ and $z$ are fixed, the ``ordered'' basic set of admissible local patterns $\mathcal{B} = \mathcal{B}_+ \bigcup \mathcal{B}_-$ must obey
$$
\mathcal{B}_+ \subseteq \{(+, -, +), (+, +, +), (+, -, -), (+, +, -)\}
$$
and
$$
\mathcal{B}_- \subseteq \{(-, +, -), (-, -, -), (-, +, +), (-, -, +)\}.
$$

\begin{figure}[tbp]
\begin{center}
\includegraphics[scale=0.8,page=6]{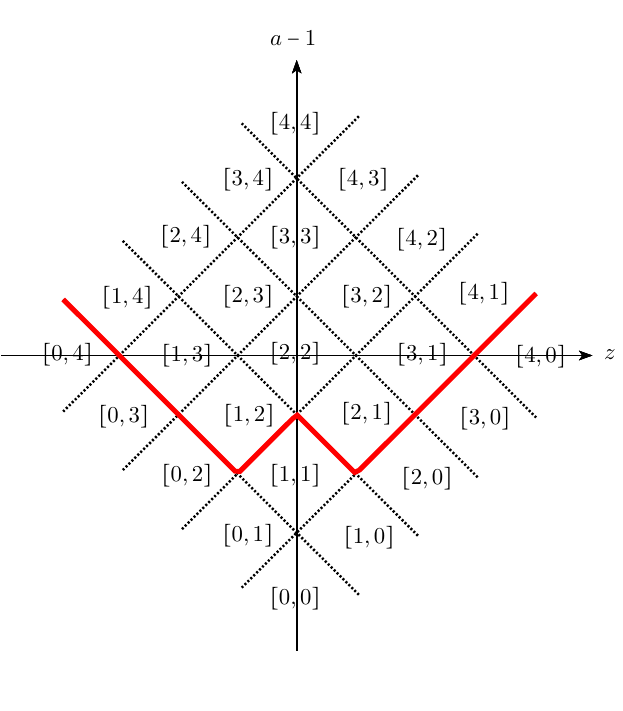}
\end{center}
\caption{Entropy bifurcation diagram of neural networks on the binary Cayley tree. The $a$-$z$ plane is partitioned into $25$ sub-regions.}
\label{fig:entropy-diagram-binary-tree}
\end{figure}

If the parameters $a$ and $z$ locate in the region $[3, 2]$(cf.~Figure \ref{fig:entropy-diagram-binary-tree}), then the basic set is
$$
\mathcal{B}_{[3, 2]} = \{(+, -, +), (+, +, +), (+, -, -), (-, +, -), (-, -, -)\}.
$$
Theorem \ref{thm:CTNN-entropy-set} and Proposition \ref{prop:entropy-region-W-equation} conclude that $h(\mathbf{Y}_{[3, 2]}) = \ln 2$ and $(a, z)$ is critical if and only if $a - 1 = \big| |z| + a_1 \big| - a_2$.
\end{example}

\section{Conclusion and Discussion}

In this paper, motivated by Gollo \emph{et al.}'s works (cf.~\cite{GKC-PCB2009,GKC-PRE2012,GKC-SR2013}), we study the dynamical behavior which tree structure neural networks are capable of. More specifically, we focus on equilibrium solutions known as mosaic solutions since they are related to the long-term memory of the brain and are applied in a wide range of disciplines. Entropy, a frequently used invariant, reveals the growth rate of the amount information stored in a (tree structure) neural network. Alternatively, positive entropy reflects that adding one more neuron stores exponential times of memory relative to the original system.  A small modification of coupling weights resulting in different entropy means the neural network is sensitive or in some critical status. We elaborate the criticality of a neural network by whether or not the neural network is entropy minimal.

After demonstrating the entropy spectrum of tree structure neural networks is discrete, we illustrate the necessary and sufficient condition for determining if a neural network is sensitive. Furthermore, the formula for coupling weights of critical neural networks is indicated.

Since the activation function considered in this article is piecewise linear transformation $f(s) = \dfrac{1}{2} (|s+1| - |s-1|)$, the output patterns of mosaic solutions are binary patterns. That is, the coloring set $\mathcal{A}$ consists of only two symbols. It is of interest that what conclusion we can derive when $\mathcal{A}$ consists of $k$ symbols for some integer $k \geq 3$. Furthermore, we focus on the rooted Cayley tree as the network's topology in the whole discussion; it is also of interest whether or not our results remain true for Bethe lattice. Related work is in preparation.

\bibliographystyle{amsplain}
\bibliography{../../grece}

\end{document}